\documentclass[10pt]{article}

\usepackage{amsmath,amssymb,amsthm}
\usepackage{MnSymbol} 
\usepackage{a4wide}

\newtheorem{Theorem}{Theorem}
\newtheorem{Corollary}[Theorem]{Corollary}

\newtheorem{Definition}[Theorem]{Definition}
\newtheorem{Example}[Theorem]{Example}
\newtheorem{Remark}[Theorem]{Remark}

\begin{document}

\title{A Non-Newtonian Noether's Symmetry Theorem\thanks{This is a preprint 
of a paper whose final and definite form is published in 'Applicable Analysis',
Print ISSN: 0003-6811, Online ISSN: 1563-504X. Submitted 02-Aug-2021, 
Accepted for publication 22-Nov-2021.}}

\author{Delfim F. M. Torres\\
\texttt{delfim@ua.pt}}

\date{R\&D Unit CIDMA, Department of Mathematics,\\ 
University of Aveiro, 3810-193 Aveiro, Portugal}

\maketitle


\begin{abstract}
The universal principle obtained by Emmy Noether in 1918,
asserts that the invariance of a variational problem with respect 
to a one-parameter family of symmetry transformations
implies the existence of a conserved quantity along
the Euler--Lagrange extremals. Here we prove Noether's 
theorem for the recent non-Newtonian calculus of variations. 
The proof is based on a new necessary optimality condition
of DuBois--Reymond type.

\medskip

\noindent \textbf{Keywords:} 
non-Newtonian calculus of variations;
Euler--Lagrange extremals; 
DuBois--Reymond condition;
Erdmann condition; 
symmetry Noether's theorem.

\medskip

\noindent \textbf{MSC 2020:} 26A24; 49K05; 49S05. 
\end{abstract}


\section{Introduction}

Recently, it has been shown that the non-Newtonian/multiplicative calculus
introduced by Grossman and Katz in \cite{MR0430173} is very useful 
in some problems of actuarial science, finance, economics, 
biology, demography, pattern recognition, signal processing, 
thermostatistics, and quantum information theory
\cite{MR2864779,e22101180,MORA20121245,MR3452941,MR4188122}.
Additionally, a mathematical problem, which is difficult or 
impossible to solve in one calculus, can be easily 
revealed through another calculus \cite{MR4247781,MR3794497}.
The literature on non-Newtonian calculus is rich and accessible,
and we assume the reader to be familiar with it. If this is not the case,
we refer, e.g., to \cite{MR4247781,MR695495,MR4188122}. Here we follow
the notations and the results published open access in \cite{delfim50}. 
We just recall: the basic four operations,
$x \oplus y = x \cdot y$;
$x \ominus y = \frac{x}{y}$;
$x \odot y = x^{\ln(y)}$;
$x \oslash y = x^{1\slash\ln(y)}$, $y \neq 1$;
the fact that $\left(\mathbb{R}^{+},\oplus,\odot\right)$ 
forms a field; and that with such arithmetic a real analysis is available,
together with all fundamental topological properties for the 
non-Newtonian metric space and a full calculus, including non-Newtonian differential
and integral equations \cite{PAP2008368,MR3073475,duyar2015some,MR3463533,MR4104356,MR4188122,MR4170025}.
We also refer the reader to the recent book \cite{book2021:Burgin:Czachor}.

Motivated by applications in economics, physics, and biology, 
in which the variational functionals to be minimized are not of the standard
form but given by a product or a quotient of integral functionals
and where admissible functions take positive values only, the author 
has recently introduced the non-Newtonian calculus
of variations, which allows to solve such multiplicative
variational problems in a rather standard way, using 
non-Newtonian variations of the form $x \oplus \epsilon \odot h$:
see \cite{delfim50}. Let $(t,x,v) \mapsto L(t,x,v)$ be a given
$C^2\left(\mathbb{R}^+\times\mathbb{R}^+\times\mathbb{R}^+;\mathbb{R}^+\right)$
function, called the Lagrangian. The fundamental problem of the non-Newtonian
calculus of variations consists to minimize the integral functional
\begin{equation}
\label{eqT:J}
\mathcal{F}\left[x(\cdot)\right] =
\strokedint_a^b L\left(t,x(t),\widetilde{x}(t)\right) \tilde{d}t
\end{equation}
over the class 
\begin{equation}
\label{eq:X}
\mathcal{X} := \left\{x \in C^2\left(\mathbb{R}^+;\mathbb{R}^+\right) : 
x(a) = \alpha,\  x(b) = \beta,\  x(t) > 0\  \forall \ t \in [a,b] \right\},  
\end{equation}
where $\widetilde{x}(t)$ is the non-Newtonian derivative of function $x$
at (positive) time $t$ and $\displaystyle \strokedint$ is the non-Newtonian (multiplicative)
integral \cite{delfim50}. The problem is solved with the help of the 
Euler--Lagrange differential equation,
\begin{equation}
\label{eqT:EL}
\frac{\tilde{d}}{\tilde{d}t}
\left[\frac{\widetilde{\partial} L}{\widetilde{\partial} v}\left(t,x(t),\widetilde{x}(t)\right)\right]
= \frac{\widetilde{\partial} L}{\widetilde{\partial} x}\left(t,x(t),\widetilde{x}(t)\right),
\end{equation}
which is a first-order necessary optimality condition \cite{delfim50}.
Each solution of \eqref{eqT:EL} is called an Euler--Lagrange extremal.

In 1918, Emmy Noether (1882--1935) established a general theorem asserting 
that the invariance of a variational integral functional under a
family of transformations depending smoothly on a real parameter $s$,
implies the existence of a conserved quantity along the Euler--Lagrange
extremals \cite{JFM46.0770.01,MR53:10538}. As corollaries, all the conservation 
laws known to classical mechanics are easily obtained. For a survey of Noether's 
theorem and its generalizations, see \cite{MR89e:49002,MR28:3353,MR2000m:49002,MR83c:70020}.
In the present article, we show that Noether's theorem is still valid 
in the non-Newtonian setting (Theorem~\ref{Th:MainResult}). 
In order to prove such result, we first prove 
a non-Newtonian DuBois--Reymond condition (Theorem~\ref{thm:1}). 
In the  particular case when the Lagrangian is autonomous, that is,
when the Lagrangian $L$ does not depend on the independent variable $t$
(i.e., $L(t,x,v) = L(x,v)$), we obtain a second Erdmann type equation
(Corollary~\ref{thm:Erd}). Two illustrative examples of application
of the obtained non-Newtonian Noether symmetry theorem are given: 
non-Newtonian conservation of energy, when the problem
is invariant under multiplicative time translations $t \mapsto t \oplus s = s \cdot t$
(Example~\ref{ex1}); non-Newtonian conservation of momentum, when the problem
is invariant under multiplicative space translations $x \mapsto x \oplus s = s \cdot x$
(Example~\ref{ex2}).


\section{The DuBois--Reymond Condition}

The importance of DuBois--Reymond optimality condition, 
and its relation with Noether's symmetry theorem, is well known 
in the literature of the calculus of variations
\cite{MR3392640,MR2098297}. We begin by proving a non-Newtonian 
DuBois--Reymond necessary optimality condition.

\begin{Theorem}[DuBois--Reymond condition]
\label{thm:1}	
If $x(t)$, $t \in [a,b]$, is a solution to problem
\begin{equation}
\label{P}
\tag{$P$}
\mathcal{F}[x] = \strokedint_{a}^{b}
L\left(t,x(t),\widetilde{x}(t)\right) \tilde{d}t 
\longrightarrow \min_{x \in \mathcal{X}},
\end{equation}
with $\mathcal{X}$ as in \eqref{eq:X},
then $x(t)$ satisfies the DuBois--Reymond condition
\begin{equation}
\label{eqT:DBR}
\frac{\widetilde{\partial} L}{\widetilde{\partial} t}\left(t,x(t),\widetilde{x}(t)\right) 
= \frac{\tilde{d}}{\tilde{d}t} \left\{
L\left(t,x(t),\widetilde{x}(t)\right)
\ominus \frac{\widetilde{\partial} L}{\widetilde{\partial} v}\left(t,x(t),\widetilde{x}(t)\right) 
\odot \widetilde{x}(t)\right\} 
\end{equation}
for all $t \in [a,b]$.
\end{Theorem}

\begin{proof}
We show the DuBois--Reymond first-order necessary optimality
condition \eqref{eqT:DBR} as a consequence of the Euler--Lagrange
equation \eqref{eqT:EL}. Let $x(t)$, $t \in [a,b]$, be a minimizer
to problem \eqref{P}. For simplicity, in what follows we omit the 
arguments, being clear that all partial derivatives of the Lagrangian 
$L$ are evaluated at $\left(t,x(t),\widetilde{x}(t)\right)$.
A direct computation shows that the right-hand side of \eqref{eqT:DBR}
is given by
\begin{equation}
\label{eq:1}
\begin{split}
\frac{\tilde{d}}{\tilde{d}t} \left\{
L \ominus \frac{\widetilde{\partial} L}{\widetilde{\partial} v} 
\odot \widetilde{x}(t)\right\} 
&= \frac{\widetilde{\partial} L}{\widetilde{\partial} t} \odot e
\oplus 
\frac{\widetilde{\partial} L}{\widetilde{\partial} x} \odot \widetilde{x}(t)
\oplus \frac{\widetilde{\partial} L}{\widetilde{\partial} v} \odot \widetilde{\widetilde{x}}(t)
\ominus \left[\frac{\tilde{d}}{\tilde{d}t}\left(
\frac{\widetilde{\partial} L}{\widetilde{\partial} v}\right) \odot \widetilde{x}(t)
\oplus \frac{\widetilde{\partial} L}{\widetilde{\partial} v} \odot \widetilde{\widetilde{x}}(t)\right]\\
&= \frac{\widetilde{\partial} L}{\widetilde{\partial} t}
\oplus
\left[\frac{\widetilde{\partial} L}{\widetilde{\partial} x} 
\ominus \frac{\tilde{d}}{\tilde{d}t}\left(
\frac{\widetilde{\partial} L}{\widetilde{\partial} v}\right)\right]
\odot \widetilde{x}(t).
\end{split}
\end{equation}
Since $x(\cdot)$ is a solution of \eqref{P},
it must satisfy the Euler--Lagrange equation
$\frac{\widetilde{\partial} L}{\widetilde{\partial} x} 
\ominus \frac{\tilde{d}}{\tilde{d}t}\left(
\frac{\widetilde{\partial} L}{\widetilde{\partial} v}\right) = 1$,
and \eqref{eq:1} simplifies to
$$
\frac{\tilde{d}}{\tilde{d}t} \left\{
L \ominus \frac{\widetilde{\partial} L}{\widetilde{\partial} v} 
\odot \widetilde{x}(t)\right\} 
= \frac{\widetilde{\partial} L}{\widetilde{\partial} t},
$$
which proves the intended result.
\end{proof}

In the autonomous case, when the Lagrangian $L$ 
does not depend on the independent variable $t$,
that is, $L(t,x,v) = L(x,v)$, we get, as
an immediate corollary of our Theorem~\ref{thm:1}, 
a non-Newtonian Erdmann necessary optimality condition.

\begin{Corollary}[Erdmann condition]
\label{thm:Erd}	
If $x(t)$, $t \in [a,b]$, is a solution to the autonomous 
variational problem
\begin{equation}
\label{Pa}
\mathcal{F}[x] = \strokedint_{a}^{b}
L\left(x(t),\widetilde{x}(t)\right) \tilde{d}t 
\longrightarrow \min_{x \in \mathcal{X}},
\end{equation}
where $\mathcal{X}$ is defined in \eqref{eq:X},
then $x(t)$ satisfies the Erdmann condition
\begin{equation}
\label{eq:Erd}
L\left(x(t),\widetilde{x}(t)\right)
\ominus \frac{\widetilde{\partial} L}{\widetilde{\partial} v}\left(x(t),\widetilde{x}(t)\right) 
\odot \widetilde{x}(t) = \text{ constant }
\end{equation}
for all $t \in [a,b]$.
\end{Corollary}

\begin{proof}
Since the Lagrangian $L$ does not depend explicitly on $t$, 
one has $\frac{\widetilde{\partial} L}{\widetilde{\partial} t} = 1$ 
and \eqref{eqT:DBR} reduces to \eqref{eq:Erd}.
\end{proof}

Corollary~\ref{thm:Erd}	already shows the importance of
DuBois--Reymond condition to establish \emph{constants of motion},
that is, quantities, like the one given by the left-hand side 
of equality \eqref{eq:Erd}, that are conserved along 
the solutions of a problem of the calculus of variations. 
The constant of motion given by the Erdmann condition
is obtained under the assumption that the Lagrangian
is autonomous, that is, time-invariant (i.e., there
exists a symmetry of the problem under time translations). 
Such relation between the invariance of the problem (or the existence
of symmetry transformations) and the existence of a constant of motion 
is the subject of our next section.


\section{Noether's Symmetry Theorem}
\label{SecT:MR}

Before formulating the Noether theorem 
for non-Newtonian extremals
of the calculus of variations, we need
to introduce a notion of invariance.
We require the symmetry transformation 
to leave the problem invariant up to first
order terms in the parameter $s$ and up to exact
differentials. Such exact differentials are 
known in physics as \emph{gauge-terms} 
\cite{MR83c:70020,MR1980565}.

\begin{Definition}[invariance]
\label{def:QIUGT}
Consider a $s$-parameter family of $C^1$-transformations
\begin{equation}
\label{eq:f:transf}
(t,x) \longrightarrow \left(T(t,x,\widetilde{x},s),X(t,x,\widetilde{x},s)\right)
\end{equation}
that reduce to the identity for $s = 1$, that is,
\begin{equation*}
T(t,x,v,1) = t, \quad X(t,x,v,1) = x,
\end{equation*}
for any $t, x, v \in \mathbb{R}^+$.
We say that the integral functional \eqref{eqT:J} is invariant under 
the one-parameter family of $C^1$-transformations \eqref{eq:f:transf}
up to the gauge-term $\Phi(t,x,\widetilde{x})$ if, and only if,
\begin{multline}
\label{eqT:defQI}
\frac{\tilde{d}}{\tilde{d}t} \Phi\left(t,x(t),\widetilde{x}(t)\right)
= \frac{\tilde{d}}{\tilde{d}s} \Biggl\{
L\Biggl(T(t,x(t),\widetilde{x}(t),s),X(t,x(t),\widetilde{x}(t),s), \\
\left.\left.\left.
\frac{\tilde{d}X}{\tilde{d}t}(t,x(t),\widetilde{x}(t),s)
\oslash \frac{\tilde{d}T}{\tilde{d}t}(t,x(t),\widetilde{x}(t),s)\right) \odot
\frac{\tilde{d} T}{\tilde{d}t}\left(t,x(t),\widetilde{x}(t),s\right)\right\}\right|_{s = 1}
\end{multline}
for all $x(\cdot) \in C^2\left([a,b]; \mathbb{R}^+\right)$.
\end{Definition}

\begin{Remark}
Similarly to the 1918 paper of Emmy Noether \cite{JFM46.0770.01,MR53:10538},
in our Definition~\ref{def:QIUGT} we consider that the derivatives of the trajectories 
$x$ may also occur in the parameter family of transformations. This possibility has been 
widely forgotten in the literature of the calculus of variations, being, however, very 
powerful from the point of view of optimal control \cite{MR1901565,delfimEJC,MR2040245}. 
The subject of non-Newtonian optimal control will be addressed in a forthcoming publication.
\end{Remark}

\begin{Theorem}[Noether's theorem]
\label{Th:MainResult}
If \eqref{eqT:J} is invariant under the one-parameter
family of time-space transformations
\begin{equation*}
(t,x) \longrightarrow \left(T(t,x,\widetilde{x},s),X(t,x,\widetilde{x},s)\right)
\end{equation*}
up to the gauge-term $\Phi\left(t,x,\widetilde{x}\right)$, then
\begin{multline}
\label{eq:N:CL}
\left[
L\left(t,x(t),\widetilde{x}(t)\right)
\ominus \frac{\widetilde{\partial} L}{\widetilde{\partial} v}\left(t,x(t),\widetilde{x}(t)\right) \odot \widetilde{x}(t)
\right]\odot 
\left.\frac{\widetilde{\partial}}{\widetilde{\partial} s} T(t,x(t),\widetilde{x}(t),s)\right|_{s = 1} \\
\oplus \frac{\widetilde{\partial} L}{\widetilde{\partial} v}\left(t,x(t),\widetilde{x}(t)\right)
\odot \left.\frac{\widetilde{\partial}}{\widetilde{\partial} s} X(t,x(t),\widetilde{x}(t),s)\right|_{s = 1}
\ominus \Phi(t,x(t),\widetilde{x}(t))
\end{multline}
is constant in $t \in [a,b]$ along any solution 
$x(\cdot) \in C^2\left([a,b]; \mathbb{R}^+\right)$
of the variational problem \eqref{P}.
\end{Theorem}

\begin{proof}
We know that if $x$ is a solution of \eqref{P}, then
it satisfies the Euler--Lagrange equation \eqref{eqT:EL} 
and the DuBois--Reymond condition \eqref{eqT:DBR}.
Having in mind that for $s = 1$ we have the identity transformation,
$T(t,x,\widetilde{x},1) = t$, $X(t,x,\widetilde{x},1) = x$, condition
\eqref{eqT:defQI} yields
\begin{equation}
\label{eqT:diffDefQI}
\begin{split}
\frac{\tilde{d}}{\tilde{d}t} \Phi\left(t,x(t),\widetilde{x}(t)\right) 
&= \frac{\widetilde{\partial} L}{\widetilde{\partial} t}\left(t,x(t),\widetilde{x}(t)\right)
\odot \left.\frac{\widetilde{\partial}}{\widetilde{\partial} s} T\left(t,x(t),\widetilde{x}(t),s\right)\right|_{s = 1} \\
&\quad \oplus \frac{\widetilde{\partial} L}{\widetilde{\partial} x}\left(t,x(t),\widetilde{x}(t)\right) \odot
\left.\frac{\widetilde{\partial}}{\widetilde{\partial} s} X\left(t,x(t),\widetilde{x}(t),s\right)\right|_{s = 1} \\
&\quad \oplus \frac{\widetilde{\partial} L}{\widetilde{\partial} v}\left(t,x(t),\widetilde{x}(t)\right) \odot
\left( \frac{\tilde{d}}{\tilde{d}t}
\left.\frac{\widetilde{\partial}}{\widetilde{\partial} s}
X\left(t,x(t),\widetilde{x}(t),s\right)\right|_{s = 1} \right.\\
&\qquad \left. \ominus \widetilde{x}(t) \odot \frac{\tilde{d}}{\tilde{d}t}
\left.\frac{\widetilde{\partial}}{\widetilde{\partial} s}
T\left(t,x(t),\widetilde{x}(t),s\right)\right|_{s = 1}\right) \\
&\quad \oplus L\left(t,x(t),\widetilde{x}(t)\right) \odot \frac{\tilde{d}}{\tilde{d}t}
\left.\frac{\widetilde{\partial}}{\widetilde{\partial} s} T\left(t,x(t),\widetilde{x}(t),s\right)\right|_{s = 1}\, .
\end{split}
\end{equation}
From \eqref{eqT:EL} one can write
\begin{multline}
\label{eqT:FromEL}
\frac{\widetilde{\partial} L}{\widetilde{\partial} x}\left(t,x(t),\widetilde{x}(t)\right) \odot
\left.\frac{\widetilde{\partial}}{\widetilde{\partial} s} X\left(t,x(t),\widetilde{x}(t),s\right)\right|_{s = 1} 
\oplus \frac{\widetilde{\partial} L}{\widetilde{\partial} v}\left(t,x(t),\widetilde{x}(t)\right) \odot
\frac{\tilde{d}}{\tilde{d}t}
\left.\frac{\widetilde{\partial}}{\widetilde{\partial} s} X\left(t,x(t),\widetilde{x}(t),s\right)\right|_{s = 1} \\
= \frac{\tilde{d}}{\tilde{d}t} \left(
\frac{\widetilde{\partial} L}{\widetilde{\partial} v}\left(t,x(t),\widetilde{x}(t)\right) \odot
\left.\frac{\widetilde{\partial}}{\widetilde{\partial} s} X\left(t,x(t),\widetilde{x}(t),s\right)\right|_{s = 1}
\right) \, ,
\end{multline}
while from \eqref{eqT:DBR} one gets
\begin{equation}
\label{eqT:FromDBR}
\begin{split}
\frac{\widetilde{\partial} L}{\widetilde{\partial} t}&\left(t,x(t),\widetilde{x}(t)\right) \odot
\left.\frac{\widetilde{\partial}}{\widetilde{\partial} s} T\left(t,x(t),\widetilde{x}(t),s\right)\right|_{s = 1}
\oplus L\left(t,x(t),\widetilde{x}(t)\right) \odot \frac{\tilde{d}}{\tilde{d}t}
\left.\frac{\widetilde{\partial}}{\widetilde{\partial} s} T\left(t,x(t),\widetilde{x}(t),s\right)\right|_{s = 1}\\
&\quad \ominus \frac{\widetilde{\partial} L}{\widetilde{\partial} v}\left(t,x(t),\widetilde{x}(t)\right) \odot \widetilde{x}(t)
\odot \frac{\tilde{d}}{\tilde{d}t}
\left.\frac{\widetilde{\partial}}{\widetilde{\partial} s} T\left(t,x(t),\widetilde{x}(t),s\right)\right|_{s = 1}\\
&= \frac{\tilde{d}}{\tilde{d}t} \left\{\left(L\left(t,x(t),\widetilde{x}(t)\right)
\ominus \frac{\widetilde{\partial} L}{\widetilde{\partial} v}\left(t,x(t),\widetilde{x}(t)\right) \odot \widetilde{x}(t)\right)
\odot \left.\frac{\widetilde{\partial}}{\widetilde{\partial} s} T\left(t,x(t),\widetilde{x}(t),s\right)\right|_{s = 1}
\right\} \, .
\end{split}
\end{equation}
Substituting \eqref{eqT:FromEL} and \eqref{eqT:FromDBR} into \eqref{eqT:diffDefQI},
\begin{multline*}
\frac{\tilde{d}}{\tilde{d}t} \left\{
\frac{\widetilde{\partial} L}{\widetilde{\partial} v}\left(t,x(t),\widetilde{x}(t)\right)
\odot \left.\frac{\widetilde{\partial}}{\widetilde{\partial} s} X(t,x(t),\widetilde{x}(t),s)\right|_{s = 1}
\ominus \Phi(t,x(t),\widetilde{x}(t)) \right. \\
\left. \oplus \left(
L\left(t,x(t),\widetilde{x}(t)\right)
\ominus \frac{\widetilde{\partial} L}{\widetilde{\partial} v}\left(t,x(t),\widetilde{x}(t)\right) \odot \widetilde{x}(t)\right)
\odot \left.\frac{\widetilde{\partial}}{\widetilde{\partial} s} T(t,x(t),\widetilde{x}(t),s)\right|_{s = 1} \right\} = 1 \, ,
\end{multline*}
and the intended conclusion is obtained.
\end{proof}


\section{Illustrative Examples}

Our first example is the analog in classical mechanics
to conservation of energy and gives an alternative proof
to the constant of motion \eqref{eq:Erd}.

\begin{Example}[conservation of energy]
\label{ex1}
Let us consider the autonomous problem of the calculus of variations
\eqref{Pa}. It is easy to see that this problem is invariant, in the sense
of Definition~\ref{def:QIUGT}, under the symmetry transformation   
$(t,x) \longrightarrow \left(T(t,s),X(x)\right)$ given by $T = t \oplus s$
and $X = x$, which for $s = 1$ reduces to  the identity transformation,
with $\Phi = \text{constant}$. Indeed, in this case
the invariance condition \eqref{eqT:defQI} is clearly satisfied:
$$
1 = \frac{\tilde{d}}{\tilde{d}s} \left\{L\left(x(t),\widetilde{x}(t)\oslash e\right) \odot e \right\}
\Leftrightarrow 1 = \frac{\tilde{d}}{\tilde{d}s} \left\{L\left(x(t),\widetilde{x}(t)\right)\right\}
\Leftrightarrow 1 = 1.
$$
It follows from Noether's theorem (Theorem~\ref{Th:MainResult}) that
\eqref{eq:N:CL} is a constant of motion, that is,
$$
L\left(x(t),\widetilde{x}(t)\right)
\ominus \frac{\widetilde{\partial} L}{\widetilde{\partial} v}\left(x(t),\widetilde{x}(t)\right) \odot \widetilde{x}(t)
$$	
is constant in $t \in [a,b]$ along any solution 
of the variational problem \eqref{Pa}.
\end{Example}

Our second example is the non-Newtonian analog 
to conservation of momentum in classical mechanics.

\begin{Example}[conservation of momentum]
\label{ex2}	
Let us now consider the following non-Newtonian
problem of the calculus of variations:
\begin{equation}
\label{P:x:absent}
\mathcal{F}[x] = \strokedint_{a}^{b}
L\left(t,\widetilde{x}(t)\right) \tilde{d}t 
\longrightarrow \min_{x \in \mathcal{X}}.
\end{equation}
In this case, since the Lagrangian $L$ does not depend on $x$,
it is trivial to see that the functional $\mathcal{F}[x]$ of \eqref{P:x:absent}
is invariant under the symmetry transformation   
$(t,x) \longrightarrow \left(T(t),X(x,s)\right)$ given by $T = t$
and $X = x\oplus s$, which for $s = 1$ reduces to  the identity transformation,
with $\Phi = \text{constant}$. It follows from 
Theorem~\ref{Th:MainResult} that
$$
\frac{\widetilde{\partial} L}{\widetilde{\partial} v}\left(t,\widetilde{x}(t)\right)
$$
is a constant of motion. 
\end{Example}

From Examples~\ref{ex1} and \ref{ex2}, we are motivated to define
the generalized momentum $p(t)$ as
\begin{equation}
\label{eq:p}
p(t) = \frac{\widetilde{\partial} L}{\widetilde{\partial} v}\left(t,x(t),\widetilde{x}(t)\right)
\end{equation}
and the Hamiltonian function $H$ by
\begin{equation}
\label{eq:H}
H(t,x,v,p) = L\left(t,x,v\right) \ominus p \odot v.
\end{equation}
With definitions \eqref{eq:p} and \eqref{eq:H}, 
the DuBois--Reymond condition \eqref{eqT:DBR}
can be written as
$$
\frac{\tilde{d}}{\tilde{d} t} \left[ H(t,x(t),\widetilde{x}(t),p(t)) \right]
= \frac{\widetilde{\partial} H}{\widetilde{\partial} t}(t,x(t),\widetilde{x}(t),p(t)) 
$$
and the Erdmann condition \eqref{eq:Erd} as
$H(x(t),\widetilde{x}(t),p(t)) = \text{ constant}$.
Such Hamiltonian perspective will be our starting point
to develop a non-Newtonian optimal control theory.
This is under development and will be addressed elsewhere.


\section{Discussion} 

In 1918, Felix Klein (1849--1925) presented the paper 
``Invariant variation problems'' by Emmy Noether (1882--1935)
\cite{JFM46.0770.01,MR53:10538} at a session of the Royal Society 
of Sciences in G\"{o}ttingen. This important paper 
is considered as a milestone in the relation between 
symmetry transformations and conservation laws in physics \cite{MR1116559}. 
Here we have proved such correspondence, between the existence
of symmetry transformations and constants of motion, in
the context of the non-Newtonian calculus of variations recently introduced
by the author in \cite{delfim50}. The main difficult is to obtain
such invariance transformations. For that, one needs to solve a PDE 
that arises from condition \eqref{eqT:defQI} of invariance.
While the resolution of such PDE has been automatized 
in the framework of the classical calculus of variations and optimal control
with the help of a computer algebra system \cite{MR2194205,MR2323264}, 
the automatic computation of conservation laws for non-Newtonian variational 
problems is still under investigation.


\section*{Acknowledgement}

This research was supported by FCT and the 
Center for Research and Development in Mathematics 
and Applications (CIDMA), project UIDB/04106/2020.



\end{document}